\documentclass{amsart}
\usepackage{amssymb, amstext, amscd, amsmath}
\usepackage{enumitem}

\theoremstyle{plain}
\newtheorem{theorem}{Theorem}[section]
\newtheorem{corollary}[theorem]{Corollary}
\newtheorem{proposition}[theorem]{Proposition}
\newtheorem{lemma}[theorem]{Lemma}
\theoremstyle{definition}

\newtheorem{remark}[theorem]{Remark}
\newtheorem{example}[theorem]{Example}

\newtheorem{definition}[theorem]{Definition}
\newtheorem{notation}[theorem]{Notation}

\newcommand{\bC}{{\mathbb{C}}}
\newcommand{\bD}{{\mathbb{D}}}

\newcommand{\bN}{{\mathbb{N}}}

\newcommand{\bT}{{\mathbb{T}}}

\newcommand{\bZ}{{\mathbb{Z}}}

\renewcommand{\phi}{\varphi}
\newcommand{\upchi}{{\raise.35ex\hbox{\ensuremath{\chi}}}}

\newcommand{\conv}{\operatorname{conv}}

\newcommand{\spn}{\operatorname{span}}

\newcommand{\supp}{\operatorname{supp}}

\newcommand{\ol}{\overline}

\newcommand{\susbeteq}{\subseteq} 

\newcommand{\cpr}{\rtimes}

\begin{document}
\title{Intermediate C$^*$-algebras of Cartan embeddings}
\author[J.H. Brown]{Jonathan H. Brown}
\address[J.H. Brown]{
Department of Mathematics\\
University of Dayton\\
300 College Park Dayton\\
OH 45469-2316 U.S.A.} \email{jonathan.henry.brown@gmail.com}

\author[R. Exel]{Ruy Exel}
\address[R. Exel]{
Departamento de Matem\'{a}tica  \\
Universidade Federal de Santa Catarina \\
Florian\'{o}polis \\
Santa Catarina\\
Brazil}
\email{r@exel.com.br}

\author[A.H. Fuller]{Adam H. Fuller}
\address[A.H. Fuller]{
Department of Mathematics\\
Ohio University\\
Athens\\
OH 45701 U.S.A.}
\email{fullera@ohio.edu}

\author[D.R. Pitts]{David R. Pitts}\thanks{This work was supported by  grants from the Simons Foundation (DRP \#316952, SAR \#36563); CNPq (RE); and by the American Institute of Mathematics SQuaREs Program.}  \address[D.R. Pitts]{
  Department of Mathematics\\
  University of Nebraska-Lincoln\\
  Lincoln\\
  NE 68588-0130 U.S.A.}  \email{dpitts2@math.unl.edu}

\author[S. A. Reznikoff]{Sarah A. Reznikoff}
\address[S. A. Reznikoff]{
Department of Mathematics\\
Kansas State University\\
138 Cardwell Hall\\
Manhattan, KS, U.S.A. }
\email{sarahrez@math.ksu.edu}


\begin{abstract}
Let $A$ be a C$^*$-algebra and let $D$ be a Cartan subalgebra of $A$.
We study the following question: if $B$ is a C$^*$-algebra such that $D \subseteq B \subseteq A$, is $D$ a Cartan subalgebra of $B$?
We give a positive answer in two cases: the case when there is a faithful conditional expectation from $A$ onto $B$, and the case when $A$ is nuclear and $D$ is a C$^*$-diagonal of $A$.
In both cases there is a one-to-one correspondence between the intermediate C$^*$-algebras $B$, and a class of open subgroupoids of the groupoid $G$, where $\Sigma \rightarrow G$ is the twist associated with the embedding $D \subseteq A$.
\end{abstract}

\dedicatory{In memory of Richard M.~Timoney and Donal O'Donovan}

\maketitle

\section{Introduction}
An interesting, and common, type of question in the study of operator algebras is the following.
Suppose $D$ is an algebra embedded in an algebra $A$ in such a way that that the inclusion has some nice properties.
If $B$ is an algebra intermediate to $D$ and $A$: i.e.\ $D \susbeteq B \subseteq A$, does the embedding of $D$ in $B$ have the same nice properties?
Is there a way to classify all the intermediate algebras $B$?

Perhaps the most famous of these results is the Galois correspondence for the crossed products of von Neumann algebra factors by discrete groups.
This says that if $N$ is a factor von Neumann algebra, and $G$ is a discrete group acting freely on $N$, then the map $H \mapsto N \cpr H$ gives a one-to-one correspondence between subgroups of $H\subseteq G$ and intermediate von Neumann algebras $N \subseteq M \susbeteq N \cpr G$.
This result is due to Izumi, Longo and Popa \cite{ILP1998}, building on the work of Choda \cite{Cho1978}.
An alternate, elegant proof was given by Cameron and Smith \cite{CamSmi2015}.

Similar Galois correspondence results have been proved for reduced crossed-products of C$^*$-algebras.
If $A$ is a C$^*$-algebra and $G$ is a discrete group acting on $A$ by automorphisms, then Choda \cite{Cho7980} gives a one-to-one correspondence between subgroups of $G$ and a certain class of intermediate C$^*$-algebras $A \subseteq B \subseteq A \cpr_r G$.
An intermediate algebra $B$ is in the desired class if, among other things,  there is a faithful conditional expectation $A \cpr_r G \rightarrow B$.
Cameron and Smith \cite{CamSmi2019}, generalising an earlier result of Izumi \cite{Izumi1998} for finite groups, showed if $A$ is a simple C$^*$-algebra and $G$ is a discrete group acting on $A$ by outer automorphisms, then the map $H \mapsto A \cpr H$ is a one-to-one correspondence between subgroups of $H \susbeteq G$ and C$^*$-algebras $B$ satisfying $A \subseteq B \subseteq A \cpr G$.
In this case there is always a faithful conditional expectation from $A \cpr_r G$ onto an intermediate C$^*$-algebra.

Beyond the rigid structure of crossed products, there are striking examples of these intermediate algebra type results.
Suppose $D$ is a Cartan subalgebra of a von Neumann algebra $M$. 
Aoi proved that if $D \subseteq N \subseteq M$, then $D$ is also a Cartan subalgebra of $N$ \cite{Aoi2003}.
Cameron, Pitts and Zarikian have given an alternative proof of this theorem \cite{CPZ2013}.
Feldman and Moore \cite{FelMooI, FelMooII} gave a one-to-one correspondence between Cartan embeddings and measured equivalence relations.
Thus, analogous to the Galois correspondence results mentioned above, Aoi's result gives a one-to-one correspondence between sub-measured equivalence relations and von Neumann algebras $D \subseteq N \subseteq M$.
Alternatively, Donsig, Fuller and Pitts \cite{DFP2017} give a measure-free description of Cartan embeddings based on extensions of inverse semigroups.
There, Aoi's theorem is used to show a one-to-one correspondence between a class of sub-inverse monoids of an inverse monoid $S$ and intermediate von Neumann algebras $D \subseteq N \subseteq M$.
These results are generalized beyond Cartan embeddings in \cite{DFP2018}.

In this note we address an analogous problem in the C$^*$-algebra setting.
Suppose $D$ is a Cartan subalgebra of a C$^*$-algebra $A$, in the sense of Renault \cite{Ren2008}.
Renault showed that there is a twist $\Sigma \overset{q}\rightarrow G$ such that $A$ can be identified with the reduced C$^*$-algebra of the twist $C_r^*(\Sigma;G)$ and $D$ is identified with $C(G^{(0)})$.
In Theorem~\ref{thm: Cartan and open sub}, we exhibit a one-to-one correspondence between open subgroupoids $H\subseteq G$ and $C^*$-algebras B such that $D\subseteq  B \subseteq A$ and $D$ is Cartan in $B$.
We further show in Theorem~\ref{thm: clopen} that this map gives a one-to-one correspondence between subgroupoids $H$ which are both closed and open in $G$ and intermediate C$^*$-algebras $D \subseteq B \subseteq A$ with a faithful conditional expectation $F \colon A \rightarrow B$.

However, a direct analogue of Aoi's theorem does not hold in this context.
There are Cartan embeddings $D \subseteq A$ with intermediate subalgebras $D \subseteq B \subseteq A$ where $D$ is not Cartan in $B$.
A relatively simple example is given in Example~\ref{ex: top free act}.
Stronger results can be found with some additional, very natural, hypotheses.

Two decades before Renault's work on Cartan subalgebras of C$^*$-algebras \cite{Ren2008}, Kumjian introduced the stronger property of C$^*$-diagonal \cite{Kum1986}.
In Theorem~\ref{thm: main} we give a Galois-correspondence type result in this context.
Suppose $D$ is a C$^*$-diagonal of a nuclear C$^*$-algebra $A$.
Let $\Sigma \overset{q}\rightarrow G$ be the twist corresponding to the embedding $D \subseteq A$.
The map $H \mapsto C_r^*(\Sigma_H;H)$ gives a one-to-one correspondence between open subgroupoids $H \subseteq G$ with $H^{(0)} = G^{(0)}$ and all C$^*$-algebras $B$ such that $D \subseteq B \subseteq A$.
In particular, if $D \subseteq B \subseteq A$ then $D$ is a C$^*$-diagonal in $B$.
In contrast with the work of \cite{Cho7980, Izumi1998, CamSmi2019} there is not necessarily a conditional expectation onto the intermediate subalgebra.

Takeishi \cite{Tak2014} showed nuclearity of $A$ is equivalent to the amenability of $G$.
Let $G^{(0)} \subseteq H \subseteq G$ be an open subgroupoid, and let $\Sigma_H \rightarrow H$ be the corresponding subtwist of $\Sigma \rightarrow G$.
In Theorem~\ref{thm: algebra support}, amenability of $G$ is used to show that, viewing $C_r^*(\Sigma_H;H)$ as a subalgebra of  $C_r^*(\Sigma;G)$, an element $a \in C_r^*(\Sigma;G)$ is in $C_r^*(\Sigma_H;H)$ if and only if $a$ (when viewed as a function on $G$) vanishes off $H$.
This is a key ingredient in the proof of Theorem~\ref{thm: main}.
When the amenability condition is removed, we show in Theorem~\ref{thm: maina} that $B$ is contained in the $C^*$-algebra of functions supported on $H$ and $C^*_r(\Sigma_H; H)\subseteq B$.

In the final section we focus on the case of C$^*$-algebraic crossed products, where we strengthen our results beyond the nuclear case.
Let $\Gamma$ be a discrete group acting on a compact Hausdorff space $X$ by homeomorphisms.
To apply Theorem~\ref{thm: main} to a crossed product we would require that the algebra $C(X)$ is a C$^*$-diagonal in the reduced crossed-product $C(X) \cpr \Gamma$, and that $C(X) \cpr \Gamma$ is nuclear.
This happens if and only if $\Gamma$ is an amenable group acting freely on $X$.
We can, however, prove a version of Theorem~\ref{thm: main} which does not require that $\Gamma$ is amenable, and instead assume $\Gamma$ has the approximation property of Cowling and Haagerup \cite{CowHaa1989}.
Let $\Gamma$ be a group which satisfies the approximation property and acts freely on a compact Hausdorff space $X$.
Let $\Gamma \times X$ be the corresponding transformation groupoid.
We prove in Corollary~\ref{cor: spectral} that the map $H \mapsto C_r^*(H)$ gives a one-to-one correspondence between open subgroupoids $H \subseteq \Gamma \times X$ with $\{e\} \times X \subseteq H$ and all C$^*$-algebras $B$ satisfying $C(X) \subseteq B \subseteq C(X) \cpr_r \Gamma$.
Thus, while there is not a Galois correspondence from the subgroups of $\Gamma$, there is a Galois-type correspondence from the open subgroupoids of the transformation groupoid $\Gamma \times X$.
This result is further evidence of the value of the groupoid approach to C$^*$-algebras, even in the relatively straightforward setting of crossed products by discrete group actions.

This Galois-type correspondence is a corollary of a spectral theorem for bimodules in $C(X) \cpr_r \Gamma$.
In Theorem~\ref{thm: spectral} we show that there is a one-to-one correspondence between the open subsets of $\Gamma \times X$ and the norm-closed $C(X)$-bimodules in $C(X) \cpr \Gamma$.
A similar spectral theorem is proved for actions of groups satisfying the approximation property on simple C$^*$-algebras in \cite{CamSmi2019}.
This result is also a direct analogue of the Spectral Theorem for Bimodules for Cartan embeddings in von Neumann algebras; see \cite{Ful1997, CPZ2013, DFP2017}.

In the von Neumann algebra setting, $L^\infty(X,\mu)$ is Cartan in the crossed product $L^{\infty}(X,\mu) \cpr \Gamma$ if and only if the action of $\Gamma$ on the measure space $(X,\mu)$ is (measurably) free \cite[Corollary~V.7.7]{TakI}.
It has been argued, e.g.\ by Tomiyama \cite{Tom1992}, that topological freeness is the correct analogue of free actions on measure spaces.
There is good reason for this viewpoint.
However, Corollary~\ref{cor: spectral} shows that there are settings when free actions, not topologically free actions, are needed in order to get desirable analogues of von Neumann algebra results.

\section{Preliminaries}
We recall the key details of Cartan embeddings in C$^*$-algebras and their relation to twists.
Recall that if $A$ is a C$^*$-algebra and $D$ is a subalgebra, then $n\in A$ \emph{normalizes} $D$ if $nDn^* \cup  n^*Dn \subseteq D$. 
We denote the normalizers of $D$ in $A$ by $N(A,D)$.
The subalgebra $D$ is \emph{regular} in $A$ if $N(A,D)$ spans a dense subset of $A$.

\begin{definition}\label{def: Cartan}
Let $A$ be a C$^*$-algebra.
A maximal abelian C$^*$-subalgebra $D \subseteq A$ is \emph{Cartan in $A$}, or a \emph{Cartan subalgebra of $A$}, if
\begin{enumerate}
\item $D$ contains an approximate identity for $A$;
\item there is a faithful conditional expectation $E \colon A \rightarrow D$;
\item $D$ is regular in $A$.
\end{enumerate}
The subalgebra $D$ is a \emph{C$^*$-diagonal of $A$} if $D$ is Cartan in $A$ and
\begin{enumerate}[resume]
\item every pure state on $D$ extends uniquely to a pure state on $A$, i.e $D$ has the \emph{unique extension property} in $A$.
\end{enumerate}
\end{definition}

It is worth noting that the conditional expectation $E$ above is unique \cite[Proposition~4.2]{Ren2008}.
Archbold, Bunce and Gregson \cite[Corollary~2.7]{ABG1982} classify precisely when an abelian subalgebra $D$ of a C$^*$-algebra $A$ satisfies the unique extension property of Definition~\ref{def: Cartan}(iv) in terms of properties of the conditional expectation.
This allows us to give the following alternative differentiation between C$^*$-diagonals and Cartan embeddings when $A$ is unital.

\begin{theorem}[c.f. {\cite[Corollary~2.7]{ABG1982}}]\label{thm: ABG}
Let $D$ be a Cartan subalgebra of a unital C$^*$-algebra $A$, with faithful conditional expectation $E \colon A \rightarrow D$.
Then $D$ is a C$^*$-diagonal of $A$ if and only if
$$ \{E(a)\} = \ol{\conv}\{ u au^* \colon u\in D, u \text{ unitary}\} \cap D, $$
for all $a\in A$.
\end{theorem}

Suppose $D$ is a Cartan subalgebra of $A$, with associated conditional expectation $E$.
We will briefly recap how to construct the groupoid twist from the embedding $D\subseteq A$.
As $D$ is abelian, $D \simeq C_0(X)$ for a locally compact Hausdorff space $X$. 
If $n \in A$ normalizes $D$ then $n^*n \in D$, and thus $n^*n$ can be viewed as a continuous positive function on $X$.
Let
$$ s(n) = \{ x\in X \colon n^*n(x) > 0 \}. $$
By \cite[Proposition~6]{Kum1986} the normalizer $n$ defines a partial homeomorphism $\beta_n$ on $X$, with domain $s(n)$ and range $s(n^*)$ satisfying
$$ (n^*dn) (x) = (dnn^*) (\beta_n(x)) $$
for all $d\in D$ and $x\in s(n)$.

For every normalizer $n\in N(A,D)$ and $x\in s(n)$, define a linear functional $[n,x]$ on $A$ by
$$ [n,x](a) = \frac{E(n^*a) (x)}{(n^*n)(x)^{1/2}}. $$
Denote by $\Sigma$ the collection of all such linear functionals $[n,x]$.
The set $\Sigma$ becomes a groupoid with partial multiplication 
\begin{equation*}
	[n,x][m,y] = [nm,y]
\end{equation*}
when $x = \beta_m(y)$; and inverse
\begin{equation*}
	[n,x]^{-1} = [n^*, \beta_n(x)].
\end{equation*}
Further, as a collection of linear functionals, $\Sigma$  inherits the relative weak$^*$-topology.
Under this topology $\Sigma$ is a topological groupoid.
There is an action of the unit circle $\bT$ on $\Sigma$ by
$$ \lambda\cdot[n,x] = [\ol{\lambda}n,x]. $$
Letting $G = \Sigma/ \bT$, we have the twist
$$ \bT \times X \hookrightarrow \Sigma \overset{q}\twoheadrightarrow G. $$
The groupoid $G$ is given the quotient topology.
Equivalently, the topology on $G$ is generated by the basic sets
$$ Z(n) = \{ q([n,x])\colon x\in s(n) \}, $$
for $n \in N(A,D)$.
Under this topology $G$ is a topologically principal, \'{e}tale groupoid \cite{Ren2008}.
If $D$ is a C$^*$-diagonal of $A$ then $G$ is a principal, \'{e}tale groupoid \cite{Kum1986}.
This twist will usually be abbreviated to $\Sigma \rightarrow G$, or $\Sigma \overset{q}\rightarrow G$ when we wish to emphasize the quotient map.

To construct the associated line bundle $L$ over $G$, we put the following equivalence relation on $\bC \times \Sigma$.
Define $(z_1,\gamma_1) \sim (z_2, \gamma_2)$ if there is a $\lambda \in \bT$ such that $(z_1,\gamma_1) = (\ol{\lambda} z_2, \lambda \cdot \gamma_2).$
Thus $[z\lambda,\gamma] = [z,\lambda\cdot\gamma]$.
Let $L = (\bC \times \Sigma)/ \sim.$
The continuous surjection $P \colon L \rightarrow G$ defined by
$P([z,\gamma]) = q(\gamma)$ makes $L$ a Fell line bundle over $G$.
For $[z,\gamma] \in L$ we define $|[z,\gamma]| = |z|$.

The collection of continuous compactly supported cross-sections $C_c(\Sigma;G)$ of $L$ becomes a $*$-algebra under the following operations.
For $f,g \in C_c(\Sigma;G)$ the convolution product on $C_c(\Sigma;G)$ is defined by
$$ f * g(\gamma) = \sum_{\gamma = \alpha \beta} f(\alpha)g(\beta), \quad \gamma \in G. $$
The adjoint is given by
$$ f^*(\gamma) = \overline{f(\gamma^{-1})}, \quad \gamma \in G.$$
The reduced C$^*$-algebra $C_r^*(\Sigma;G)$ of the twist $\Sigma \rightarrow G$ is the unique closure of the $*$-algebra $C_c(\Sigma;G)$ so that the restriction map $E \colon C_c(\Sigma;G) \rightarrow C_c(G^{(0)})$ extends to a faithful conditional expectation $E \colon C_r^*(\Sigma;G) \rightarrow C_0(G^{(0)})$.
A detailed description of the construction of the reduced C$^*$-algebra $C_r^*(\Sigma;G)$ can be found in \cite[Section~2]{BFPR}.

We recall main theorems of \cite{Kum1986} and \cite{Ren2008}.

\begin{theorem}[c.f. \cite{Kum1986} and \cite{Ren2008}] \label{thm: Kumjian Renault}
Let $D$ be a Cartan subalgebra of a C$^*$-algebra $A$.
Then $A$ is isomorphic to $C_r^*(\Sigma;G)$ via an isomorphism which carries $D$ to $C_0(G^{(0)})$.

Conversely if $\Sigma \rightarrow G$ is a twist with $G$ topologically principal and \'{e}tale then $C_0(G^{(0)})$ is a Cartan subalgebra of $C_r^*(\Sigma;G)$.
The algebra $C_0(G^{(0)})$ is a C$^*$-diagonal of $C_r^*(\Sigma; G)$ if and only if $G$ is a principal groupoid.
\end{theorem}

\begin{remark}
If $D$ is a Cartan subalgebra of $A$ with associated twist $\Sigma \rightarrow G$, Theorem~\ref{thm: Kumjian Renault} allows us to identify $A$ with $C_r^*(\Sigma;G)$ and $D$ with $C_0(G^{(0)}$.
In the sequel we will do this without comment.
\end{remark}

Our main result, Theorem~\ref{thm: main},  applies when the groupoid is amenable.
Amenability of groupoids was introduced by Renault \cite[Chapter~II.3]{RenaultBook}.
For our purposes we will use a characterization of amenability in terms of positive-type functions.
Let $G$ be a groupoid.
A function $h \colon G \rightarrow \bC$ is of \emph{positive-type} if for all $x\in G^{(0)}$ and finite subsets $F \subseteq G_x$ the matrix $[h(\eta^{-1}\gamma)]_{\eta,\gamma \in F}$ is a positive matrix.

\begin{theorem}[{c.f.~\cite[Theorem~5.6.18]{BrownOzawa}}]\label{thm: amenable groupoid}
  Let $G$ be an \'etale locally compact groupoid.  Then $G$ is amenable
  if and only if there is a net of $\sup$-bounded, positive-type
  functions in $C_c(G)$ which converges uniformly to $1$ on compact
  subsets of $G$.
\end{theorem}

We note the following theorem of Takeishi \cite{Tak2014} which classifies when $C_r^*(\Sigma;G)$ is nuclear.
The case when the twist is trivial can be found in \cite[Theorem~5.6.18]{BrownOzawa}, and for more general groupoids in \cite[Corollary~6.2.14]{AmaRenbook}.

\begin{theorem}[{c.f.~\cite[Theorem~5.4]{Tak2014}}]\label{thm: nuclear and amenable}
Let $\Sigma \rightarrow G$ be a twist, with $G$ an \'etale locally compact groupoid. 
Then $C_r^*(\Sigma;G)$ is nuclear if and only if $G$ is amenable.
\end{theorem}

One can also construct the \emph{full} C$^*$-algebra of the twist $C^*(\Sigma;G)$ by completing $C_c(\Sigma;G)$ with respect to the supremum norm over all $*$-representations of $C_c(\Sigma;G)$.
If $f \in C_c(\Sigma;G)$ is supported on an open bisection $U \subseteq G$, then $f^* * f \in C_c(G^{(0)})$ and $\|f\|_{C^*(\Sigma;G)} \leq \|f\|_\infty$.
It follows that for each compact $K \subseteq G$ there is a $C_K>0$ such that if $f \in C_c(\Sigma;G)$ and $f$ is supported on $K$, then 
\begin{equation}\label{eq: full norm} 
	\|f\|_{C^*(\Sigma;G)} \leq C_k \|f\|_\infty;
\end{equation}
see \cite[page~205]{BrownOzawa}.
We note the following result due to Sims and Williams \cite{SimWil2013}.

\begin{theorem}[{cf.~\cite[Theorem~1]{SimWil2013}}]\label{thm: reduced eq full}
Let $\Sigma \rightarrow G$ be a twist, with $G$ an \'etale, Hausdorff, amenable groupoid.
Then the reduced C$^*$-algebra $C_r^*(\Sigma;G)$ is isomorphic to the full C$^*$-algebra $C^*(\Sigma;G)$.
\end{theorem}

\section{Intermediate subalgebras of Cartan embeddings}
Our primary goal is to study C$^*$-algebras $B$ satisfying $D \subseteq B \subseteq A$, when $D$ is Cartan in $A$.
In particular, does the (twisted) groupoid C$^*$-algebra structure of $A$ force $B$ to have a similar structure?
The main result of this section, Theorem~\ref{thm: clopen}, gives a positive answer in the case when there is a faithful conditional expectation from $A$ to $B$.
\begin{notation}
If $\Sigma \overset{q}\rightarrow G$ is a twist and $H \subseteq G$, we let $\Sigma_H = q^{-1}(H)$.
\end{notation}

We recall the following result.

\begin{lemma}[{c.f.\ \cite[Lemma~2.3]{BFPR}}]\label{lem: BFPR}
Let $\Sigma \overset{q}\rightarrow G$ be a twist.
Let $H$ be an open subgroupoid of $G$, with $G^{(0)} \subseteq H$.
Then $\Sigma_H \rightarrow H$ is a twist.
Further, the map of extension by zero of $C_c(\Sigma_H;H) \hookrightarrow C_c(\Sigma;G)$ induces an embedding $\iota$ of $C^*_r(\Sigma_H;H)$ into $C^*_r(\Sigma;G)$.
\end{lemma}

We now describe the Cartan intermediate subalgebras.

\begin{theorem}\label{thm: Cartan and open sub}
Let $D$ be Cartan in $A$, with twist $\Sigma \rightarrow G$.
Let $D \subseteq B \subseteq A$ be an intermediate subalgebra.
Then $D$ is Cartan in $B$ if and only if there exists an open subgroupoid $H \subseteq G$ with $H^{(0)} = G^{(0)}$ and $B = \iota(C_r^*(\Sigma_H;H))$, where $\iota$ is the inclusion from Lemma~\ref{lem: BFPR}.

Further, in this case $H = \{\gamma \in G \colon q^{-1}(\gamma)|_B \neq 0\}$.
\end{theorem}

\begin{proof}
Lemma~\ref{lem: BFPR} shows that when $H$ is an open subgroupoid $D$ is a Cartan subalgebra of $C_r^*(\Sigma_H;H)$.

Suppose $D \subseteq B \subseteq A$ with $D$ Cartan in $B$.
Take $[n,x]\in \Sigma$ such that $q([n,x]) \in H$.
Since we are assuming that $D$ is regular in $B$, it follows that there is a normalizer $m \in N(B,D)$ such that $[n,x](m) \neq 0$.
That is,
$$ \frac{E(n^*m)(x)}{n^*n(x)^{1/2}} \neq 0. $$
As $E(n^*m)(x) \neq 0$ it follows that $[n,x] = [m,x]$ \cite[Lemma~8.7]{PitStrArx}.
As $N(B,D)$ is a $^*$-semigroup, it follows that $\Sigma_H$ and $H$ are groupoids.
That $H$ is open follows from the definition of the topology on $G$.

Finally, that $B = C_r^*(\Sigma_H;H)$ follows from Theorem \ref{thm: Kumjian Renault}.
\end{proof}

Given C$^*$-algebras $D \subseteq B \subseteq A$ with $D$ Cartan in $A$, it is not, in general, easy to tell whether $D$ is Cartan in $B$.
Indeed it may be not be.
See Example~\ref{ex: top free act}.
We will now give a class of intermediate C$^*$-algebras $B$ where we are guaranteed that $D$ is Cartan in $B$.
These will be in one-to-one correspondence with the subgroupoids $G^{(0)} \subseteq H \subseteq G$ with $H$ both open and closed in $G$.

In the Galois correspondence results in \cite{Cho7980, Izumi1998, CamSmi2019} there is always a conditional expectation onto the intermediate algebras $A \subseteq B \subseteq A\rtimes_r \Gamma$.
In the following theorem we classify the intermediate subalgebras  $B$ in Cartan embeddings for which there is a conditional expectation onto $B$.

\begin{lemma}\label{lem: clopen}
Let $\Sigma \rightarrow G$ and let $H$ an open subgroupoid of $G$.
Then there is a conditional expectation $F \colon C_r^*(\Sigma;G) \rightarrow C_r^*(\Sigma_H;H)$ if and only if $H$ is closed.
\end{lemma}

\begin{proof}
We will follow a similar line of proof as in \cite[Proposition~4.1]{BNRSW2016}.
Assume $H$ is not closed.
Then there is a net $(\gamma_n)$ in $H$ which converges to some $\gamma \in G\backslash H$.
Let $U\in G$ be an open bisection containing $\gamma$ and choose $f \in C_c(\Sigma;G)$ supported on $U$ such that $|f(\gamma)| =1$.
As $\gamma_n\rightarrow \gamma$, we may assume that each $\gamma_n \in U$.
Since $H$ is open, there are open sets $V_n \subseteq H \cap U$ such that $\gamma_n \in V_n$.
For each $n$ we choose $g_n \in C_c(G^{(0)})$ with $g_n$ supported on $r(V_n)$ and $g_n(r(\gamma_n)) = 1$.

Since for each $n$, $g_n \in C_c(G^{(0)}) \subseteq C_c(\Sigma_H;H)$, it follows that
$$ F(g_n * f) = g_n * F(f). $$
Further, note $g_n * f$ is supported on $V_n \subseteq H$.
Thus $g_n * f \in C_c(\Sigma_H;H)$ and $F(g_n * f) = g_n * f$.
Finally, note that for any $h \in C_r^*(\Sigma;G)$, $g_n * h(\gamma_n) = h(\gamma_n)$, since $g_n(r(\gamma_n)) =1$.
Hence
$$ F(f) (\gamma_n) = g_n* F(f)(\gamma_n) = F(g_n * f)(\gamma_n) =g_n * f(\gamma_n) = f(\gamma_n).$$
Since $|f(\gamma) = 1|$, it follows that $|F(f) (\gamma_n)| \rightarrow 1$.
Thus $|F(f)(\gamma)| =1$. 
However, since $\gamma \notin H$, and $F(f) \in C_c(\Sigma_H;H)$, we must have $F(f)(\gamma) = 0$.
This is a contradiction, and hence $H$ is closed.

Conversely, if $H$ is both open and closed in $G$ then the restriction map induces a faithful conditional expectation from $C_r^*(\Sigma;G)$ onto $C_r^*(\Sigma_H;H)$.
\end{proof}

\begin{theorem}\label{thm: clopen}
Let $A$ be a C$^*$-algebra and let $D \subseteq A$ be a Cartan subalgebra.
Let $\Sigma \rightarrow G$ be the corresponding twist.
The following statements hold.
\begin{enumerate}
\item Let $D \subseteq B \subseteq A$ be an intermediate subalgebra. If there is a conditional expectation $F \colon A \rightarrow B$ then $D$ is Cartan in $B$.
\item The map $H \mapsto C_r^*(\Sigma_H;H)$ gives a one-to-one correspondence between subgroupoids $G^{(0)} \subseteq H \subseteq G$ which are both open and closed in $G$, and the intermediate C$^*$-algebras $D \subseteq B \subseteq A$ for which there is a  conditional expectation $F \colon A \rightarrow B$.
\end{enumerate}
\end{theorem}

\begin{proof}
Let $B$ be a C$^*$-algebra such that $D \subseteq B \subseteq A$ and
there is a conditional expectation $F \colon A \rightarrow B$. 
Let $E \colon A \rightarrow D$ be the unique faithful conditional expectation.
Then $(E|_B)\circ F=E$, from which it follows that $F$ is
necessarily faithful.

Let $n\in A$ be an intertwiner of $D$. 
That is, for each $d \in D$ there is a $d'\in D$ such that $nd= d'n$.
It follows that for $d\in D$
$$ F(n)d = F(nd)= F(d'n) = d'F(n). $$
Hence $F(n)$ is also an intertwiner of $D$.  By
\cite[Proposition~3.3]{DonPit2008} the normalizers $N(A,D)$ are the
closure of all intertwiners.  Since $F$ is a continuous map, it follows
that if $n \in N(A,D)$ is a normalizer then $F(n) \in N(B,D)$.  Since $D$ is regular in $A$, if $b \in B$ there
is a sequence $(a_n)_n \in \spn N(A,D)$ such that $a_n \rightarrow b$.  Hence
$F(a_n) \rightarrow F(b)=b$.  It follows that $\spn N(B,D)$ is dense
in $B$, so $D$ is Cartan in $B$.

By Theorem~\ref{thm: Cartan and open sub}, since $D$ is Cartan in $B$, there is an open subgroupoid $H \subseteq G$ with $H^{(0)} = G^{(0)}$ such that $C_r^*(\Sigma_H;H)$. 
By Lemma~\ref{lem: clopen}, $H$ is also closed.
\end{proof}

\section{Nuclear C$^*$-algebras and intermediate algebras of C$^*$-diagonals}
In this section we consider C$^*$-algebras $A$ containing a C$^*$-diagonal $D$, with associated twist $\Sigma \rightarrow G$.
We will see that if $D \subseteq B \subseteq A$ and $A$ is nuclear, then $D$ is necessarily a C$^*$-diagonal in $B$.
This will gives a one-to-one correspondence between open subgroupoids of $G$ and the intermediate C$^*$-algebras $B$, Theorem~\ref{thm: main}.

The following theorem tells us that, in the amenable case, determining whether or not an element $a\in C_r^*(\Sigma;G)$ lies in $C_r^*(\Sigma_H;H)$ or not, depends solely on the support of $a$ as a function in $C_0(\Sigma;G)$.

\begin{notation}
Let $\Sigma \rightarrow G$ be a twist.
If $H \subseteq G$ is an open subgroupoid we let
$$ A_H = \{ f\in C_r^*(\Sigma;G) \colon \text{for all } \gamma \in \Sigma\backslash\Sigma_H\ \gamma(f) = 0 \}.$$ 
\end{notation}

In general $C_r^*(\Sigma_H;H) \subseteq A_H$.
We suspect this containment may be strict in some circumstances.

\begin{theorem}\label{thm: algebra support}
Let $\Sigma \rightarrow G$ be a twist with $G$ an \'etale, Hausdorff, amenable groupoid.
Let $H$ be an open subgroupoid of $G$.
Then $A_H$ is equal to the canonical copy (as in Lemma~\ref{lem: BFPR}) of $C_r^*(\Sigma_H;H)$ within $C_r^*(\Sigma;G)$.
\end{theorem}

\begin{proof}
By Theorem~\ref{thm: amenable groupoid}, there is a net of positive-type functions $h_i \in C_c(G)$ converging uniformly to $1$ on compact subsets of $G$, with $\sup_{\gamma \in G}|h_i(\gamma)| \leq 1$.
View $C_c(\Sigma;G)$ as continuous cross-sections.
For each $i$, define a map
\begin{align*}
	m_{h_i} \colon  C_c(\Sigma;G) &\rightarrow C_c(\Sigma;G) \\
	 f &\mapsto h_if,
\end{align*}
where for $\gamma \in G$ and $f(\gamma)= [\lambda_\gamma,\gamma]$ we have $(h_if)(\gamma) = [\lambda_\gamma h_i(\gamma),\gamma]$.
By \cite[Lemma~4.2]{Tak2014}, $m_{h_i}$ extends to a contractive completely positive map on $C_r^*(\Sigma;G)$.
As $G$ is amenable, $C_r^*(\Sigma;G) = C^*(\Sigma;G)$, by Theorem~\ref{thm: reduced eq full}.
Note that $m_{h_i}(f)$ converges to $f$ in $C^*(\Sigma;G)=C_r^*(\Sigma;G)$, by Equation~\eqref{eq: full norm}.
We further note that, for any $f \in C_r^*(\Sigma;G)$, $m_{h_i}(f) \in C_c(\Sigma;G)$.

Suppose $f\in A_H$. 
To show that $f \in C_r^*(\Sigma_H;H)$ it suffices to show that $m_{h_i}(f) \in C_r^*(\Sigma_H;H)$ for every $i$.
As $m_{h_i}(f)$ lies in $C_c(\Sigma;G)$, and vanishes off $H$, $m_{h_i}(f) \in C_c(\Sigma_H;H)$.
It follows that $f \in C_r^*(\Sigma_H;H)$, concluding the proof.
\end{proof}

We recall the following corollary to Theorem~\ref{thm: ABG}.
It was observed in the unital case by Muhly, Qiu and Solel  in the proof of Proposition~4.4 of \cite{MQS1991}.

\begin{proposition}\label{prop: fourier coefs exist}
Suppose $D \subseteq A$ is a C$^*$-diagonal.
Then, for any $a\in A$ and normalizer $n \in N(A,D)$, $nE(n^*a)$ is a normalizer of $D$ in the norm-closed $D$-bimodule generated by $a$.
\end{proposition}

\begin{proof}
If $A$ is unital the result follows from Theorem~\ref{thm: ABG}.
Details can be found in \cite[Proposition~3.1]{DonPit2008}.

Suppose now that $A$ is not unital.
Let $\tilde{A}$ be the unitization of $A$.
It can be shown that $D + \bC I$ is a maximal abelian subalgebra of $\tilde{A}$ and $D + \bC I$ has the unique extension property in $\tilde{A}$.
By \cite[Corollary~2.7]{ABG1982} there is a conditional expectation $\tilde{E} \colon \tilde{A} \rightarrow D + \bC I$ determined by
$$ \tilde{E}(x) = \overline{\conv} \{uxu^* \colon u \in D + \bC I \text{ unitary} \} \cap (D + \bC I). $$
By \cite[Remarks~2.8(iii)]{ABG1982}, the restriction of $\tilde{E}$ to $A$ maps $A$ onto $D$.
As the conditional expectation $E \colon A \rightarrow D$ is unique \cite[Proposition~4.3]{Ren2008}, $\tilde{E}|_A = E$.

Note that if $n\in N(A, D)$ then $n \in N(\tilde{A}, D + \bC I)$.
It follows now as in the unital case, that for any $n \in N(A, D)$, $nE(n^*a)$ normalizes $D$ and is in the norm-closed $D$-bimodule generated by $A$.
\end{proof}

We now have all the ingredients for the main theorems of this section.

\begin{theorem}\label{thm: maina}
Let $D$ be a C$^*$-diagonal of $A$ with associated twist $\Sigma \rightarrow G$.
Let $D \subseteq B \subseteq A$ be an intermediate subalgebra.
Let $H = \{ q(\sigma) \colon \sigma \in \Sigma,\ \sigma|B\neq 0\}.$
Then $H$ is an open subgroupoid of $G$ and
$$ \ol{\spn} N(B,D) = C_r^*(\Sigma_H;H) \subseteq B \subseteq A_H. $$
\end{theorem}

\begin{proof}
Take $\sigma \in \Sigma$ and $b\in B$ such that $\sigma(b) \neq 0$.
By definition of $\Sigma$, there is an $n\in N(A,D)$ and $x$ in  $\hat D$ such that $\sigma = [n,x]$.
Thus we have
\begin{align*}
	0\neq [n,x](b) = \frac{E(n^*b)(x)}{(n^*n)(x)^{1/2}}.
\end{align*}	
Thus $E(n^*b)\neq 0$, and hence $nE(n^*b)\neq 0$.
By Proposition~\ref{prop: fourier coefs exist}, $nE(n^*b) \in B \cap N(A,D)$.
Let $m = nE(n^*b)$.
For any $a \in A$
\begin{align*}
	[m,x](a) & = \frac{E(m^*a)(x)}{(m^*m)(x)^{1/2}} 
	= \frac{ E(n^*a)(x)}{ |E(n^*a)(x)|} \sigma(a)= \sigma(a).
\end{align*}
Thus $q([m,x]) = q(\sigma)$.
Thus
\begin{equation}\label{eq: H sub}
 H = \{ q([m,x]) \colon m \in N(B) \}, 
\end{equation}
and hence $H\subseteq G$ is an open subgroupoid.
Let $B_0 = \ol{\spn} \{ n \in N(B,D)\}$.
Then $B_0$ is a C$^*$-algebra satisfying $B_0 \subseteq B$.
By definition, $D$ is regular in $B_0$, and thus $B_0 = C_r^*(\Sigma_H;H)$, by Equation~\eqref{eq: H sub}, Theorem~\ref{thm: Cartan and open sub} and Theorem~\ref{thm: Kumjian Renault}.
\end{proof}

\begin{theorem}\label{thm: main}
Let $A$ be a nuclear C$^*$-algebra, and let $D$ be a C$^*$-diagonal of $A$.
Let $\Sigma \rightarrow G$ be the corresponding twist for the embedding $D \subseteq A$.

Then there is a one-to-one correspondence between C$^*$-algebras $B$ satisfying $D \subseteq B \subseteq A$ and open subgroupoids $H$ satisfying $G^{(0)} \subseteq H \subseteq G$, given by an isomorphism $B \simeq C_r^*(\Sigma_H;H)$.
\end{theorem}

\begin{proof}
The result follows from Theorem~\ref{thm: algebra support} and Theorem~\ref{thm: maina}.
\end{proof}

\begin{corollary}
Let $A$ be a nuclear C$^*$-algebra, and let $D$ be a C$^*$-diagonal of $A$.
If $B$ is a C$^*$-algebra satisfying $D \subseteq B \subseteq A$, then $D$ is a C$^*$-diagonal of $B$.
\end{corollary}

\section{Crossed Products}
Let $\Gamma$ be a discrete group acting on a compact Hausdorff space $X$ by homeomorphisms.
The set $\Gamma \times X$ becomes the \emph{transformation groupoid} under the partial multiplication
$$ (g_1,g_2 \cdot x)(g_2,x) = ((g_1g_2), x)$$
and inversion
$$ (g,x)^{-1} = (g^{-1},g\cdot x). $$
With the product topology the transformation groupoid $\Gamma \times X$ is \'etale, and clearly Hausdorff.
Let $C(X) \cpr_r \Gamma$ be the reduced crossed product.
One can show that $C(X) \cpr_r \Gamma = C_r^*(\Gamma \times X)$.

We view $C(X)$ as a subalgebra of $C(X) \cpr_r \Gamma$ and denote the canonical unitary representation of $\Gamma$ in $C(X) \cpr_r \Gamma$ by $\{u_g\}_{g\in \Gamma}$.
Denote by $E$ the usual faithful conditional expectation from $C(X) \cpr_r \Gamma$ to $C(X)$; see \cite[Proposition~4.1.9]{BrownOzawa}.
Recall that the action of $\Gamma$ on $X$ is \emph{free} if for each $g\neq e$ in $\Gamma$ the set $\{g\cdot x = x \colon x\in X\}$ is empty.
The action is \emph{topologically free} if for each $g\neq e$ in $\Gamma$ the set $\{g \cdot x = x \colon x\in X\}$ has empty interior.
Zeller-Meier \cite{ZM1968} showed that $C(X)$ is maximal abelian in $C(X) \cpr_r \Gamma$ if and only if the action of $\Gamma$ on $X$ is topologically free.
Thus, $C(X)$ is Cartan in $C(X) \cpr_r \Gamma$ if and only if the action of $\Gamma$ on $X$ is topologically free.
Further $C(X)$ is a C$^*$-diagonal in $C(X) \cpr_r \Gamma$ if and only if the action of $\Gamma$ on $X$ is free.
This can be seen by the fact that the transformation groupoid $\Gamma \times X$ is principal if and only if the action of $\Gamma$ is free.
One can alternatively show that freeness of the action of $\Gamma$ implies that
$$ E(a) \in \ol{\conv}\{ uau^* \colon u \in C(X),\ u\text{ unitary}\}, $$
see \cite[Proposition~11.1.19]{GKPTbook}.
Thus $C(X)$ has the unique extension property in $C(X) \cpr_r \Gamma$ by Theorem~\ref{thm: ABG} if and only if the action of $\Gamma$ on $X$ is free.
It follows that if $\Gamma$ acts by a topologically free action on $X$ then $C(X)$ is Cartan in $C(X) \rtimes \Gamma$ with associated twist being simply the trivial twist on the transformation groupoid $\Gamma \times X$.

Consider a reduced crossed product $C(X) \cpr_r \Gamma$, where
$\Gamma$ is a discrete group acting on a compact Hausdorff
space $X$.  Recall that every element $a \in C(X) \cpr_r \Gamma$ is
uniquely determined by its Fourier series.  We write
$$ a \sim \sum_{g\in \Gamma} a_g u_g $$
where $a_g = E(au_g^*) \in C(X)$.  To be wholly consistent with the
general theory of C$^*$-diagonals discussed above one would consider
the elements $u_gE(u_g^*a)$, as in Proposition~\ref{prop: fourier
  coefs exist}, but we will stick the usual convention of setting
$a_g = E(au_g^*)$ when working with crossed products.
Note that, if the action of $\Gamma$ is free then, as
in Proposition~\ref{prop: fourier coefs exist}, $a_gu_g$ is in the
norm-closed $C(X)$-bimodule generated by $a$.

The elements of the transformation groupoid $\Gamma \times X$ act as linear functionals on $C(X) \cpr_r \Gamma$. 
As our Fourier coefficients are of the form $E(au_g^*)u_g$ and \emph{not} $u_gE(u_g^* a)$, we define the linear functionals slightly differently than in the general Cartan/C$^*$-diagonal case, with the normalizers acting on the other side.
Thus, for $a \in C(X) \cpr_r \Gamma$ and $(g,x) \in \Gamma \times X$ we have
$$ (g,x) (a) = E(au_g^*)(x) = a_g(x), $$
when $a$ has Fourier series $a \sim \sum_{g\in\Gamma} a_g u_g$.

The following example shows that Theorem~\ref{thm: main} need not hold for Cartan embeddings even when $C(X) \cpr_r \Gamma$ is nuclear.

\begin{example}\label{ex: top free act}
Let $\ol{\bD}$ be the closed unit disk in $\bC$.
Let $\bZ$ act on $\ol{\bD}$ by an irrational rotation.
That is, we fix $\alpha\in \bT$ such that $\{\alpha^n: n\in\bN\}$ is
infinite and for $n\in \bZ$ and $z \in \ol\bD$, we define
\[ n \cdot z = \alpha^n z.\]
Then the action of $\bZ$ on $\ol\bD$ is topologically free.
It is not a free action since $0$ is fixed.
Hence $C(\ol\bD)$ is Cartan in $C(\ol\bD)\cpr \bZ$, but it is not a C$^*$-diagonal.
Let $u \in C(\ol\bD) \cpr \bZ$ be the unitary generating the
representation of $\bZ$.

Consider the ideal
$$ J = \{ f\in C(\ol\bD) \colon f(0) = 0 \}$$
in $C(\ol\bD)$.
Since $\{0\}$ is fixed under our action, the set
$$ K = \ol{\spn}\{ fu^n \colon f\in J \} $$
is an ideal in $C(\ol\bD) \cpr \bZ$, \cite[Proposition~VIII.3.3]{DavCstarBook}.
Let $q$ denote the quotient map
\[ q \colon C(\ol\bD) \cpr \bZ \rightarrow (C(\ol\bD)\cpr\bZ)/ K
\simeq C(\bT); \]  indeed,  for $N\in\bN$, \[ q\left(\sum_{n=-N}^N
f_nu^n\right)=\sum_{n=-N}^N f_n(0)z^n\in C(\bT).\] 
Let
$$ Y = \{ f \in C(\bT) \colon f(1) = f(-1)\}$$
and let $B = q^{-1}(Y)$. 
Then $B$ is a C$^*$-algebra satisfying
$$ C(\ol\bD) \subseteq B \subseteq C(\ol\bD) \cpr \bZ. $$

For any
$x\in \bD$ and $r\in \bZ$, a calculation shows that for $f\in
C(\overline \bD)$ and $n\in \bZ$,
\[(r,x)(fu^n)=\begin{cases} 0& \text{if $r\neq n$};\\  f(r\cdot x) &
    \text{when $r=n$.}
  \end{cases}\]

Suppose $r\in \bZ$ is odd.  Let $b= u^r-u^{-r}\in B$.
Then for any $x\in \ol{D}$, $(r,x)(b)\neq 0$.
On the other hand, when $r$ is even, $u^r\in B$, so for  $x\in \ol{\bD}$,
$(x,r)(u^r)\neq 0$.   
Hence, $\{\sigma\in \bZ \times \ol{\bD}: \sigma|_B\neq 0\}=\bZ \times \ol{\bD}$.
Theorem~\ref{thm: Cartan and open sub} thus implies that
$C(\ol{\bD})$ is not Cartan in $B$.
\end{example}

Example~\ref{ex: top free act} shows that in general there is not a nice correspondence between subgroupoids and intermediate C$^*$-algebras of a Cartan embedding.
It may be that Theorem~\ref{thm: clopen} is the best we can hope for in that generality.
The following corollary to Theorem~\ref{thm: clopen} recovers a special case of Choda's results \cite{Cho7980}.

\begin{corollary}
Let $\Gamma$ be a discrete group acting topologically freely on  a connected Hausdorff space $X$.
Then the map $H \mapsto C(X) \cpr_r H$ gives a one-to-one correspondence between subgroups $H \subseteq \Gamma$ and intermediate C$^*$-algebras $C(X) \subseteq B \subseteq C(X) \cpr_r \Gamma$ for which there is a faithful conditional expectation $C(X) \cpr_r \Gamma \rightarrow B$.
\end{corollary}

\begin{proof}
Since $X$ is connected the only subgroupoids of $\Gamma \times X$ which are both closed and open are of the form $H \times \Gamma$ where $H \subseteq \Gamma$ is a subgroup.
Thus, the result follows from Theorem~\ref{thm: clopen}.
\end{proof}

We now turn our attention to free actions.
To apply Theorem~\ref{thm: main} directly we also need to know when $C(X) \cpr_r \Gamma$ is nuclear.
This happens precisely when the action of $\Gamma$ on $X$ is amenable, which in turn happens precisely when the transformation groupoid $\Gamma \times X$ is amenable.
Thus, for $C(X)$ to be a C$^*$-diagonal in a nuclear $C(X)\cpr_r \Gamma$ we need the action of $\Gamma$ on $X$ to be free and amenable.
However, if the action of $\Gamma$ on $X$ is both free and amenable, then $\Gamma$ itself must be amenable \cite[Corollary~4.3]{Moo2009}.
Hence Theorem~\ref{thm: main}, when applied to crossed products, applies only to free actions of amenable groups.
We can loosen this condition.

If $a \in C(X) \rtimes \Gamma$ with Fourier series $a \sim  \sum_{g\in \Gamma} a_g u_g$.
One cannot expect the series
$\sum_{g\in \Gamma} a_g u_g$ to converge in norm (though it does 
converge in the Bures-topology in the enveloping von Neumann algebra
\cite{Mer1985}).  With the right class of groups $\Gamma$ we can,
however, recover any element $a\in C(X) \cpr_r \Gamma$ from its
Fourier series.

\begin{definition}
A discrete group $\Gamma$ has the \emph{approximation property} if $1$
is in the weak$^*$-closure of the space of finitely supported
functions on $\Gamma$.\footnote{To be explicit, $\Gamma$ has the
  approximation property if there exists a net $(f_\alpha)$ in
  $c_c(\Gamma)$ such that for every $\phi\in \ell^1(\Gamma)$,
  $\sum_{\gamma\in \Gamma} \phi(\gamma) f_\alpha(\gamma)\rightarrow
  \sum_{\gamma\in \Gamma} \phi(\gamma)$.}
\end{definition}

The approximation property was introduced by Cowling and Haagerup in \cite{CowHaa1989}, where they showed that both amenable and weakly amenable groups satisfy the approximation property.
The following result was proved by B\'edos and Conti \cite[Section~5]{BedCon2015}, generalizing earlier work of Exel \cite{Exe1997}.

\begin{theorem}\label{thm: ap fourier}
Let $\Gamma$ be a discrete group satisfying the approximation property and let $C(X) \cpr_r \Gamma$ be a reduced crossed product.
If $a\in C(X) \cpr_r \Gamma$ has Fourier series $a \sim \sum_{g\in \Gamma} a_g u_g$, then
$$ a \in \overline{\spn}\{a_gu_g \colon g\in \Gamma\}.$$
\end{theorem}

\begin{remark}
Theorem~\ref{thm: ap fourier} applies to twisted crossed products.
Thus, the results which follow will also apply in that generality.
\end{remark}

\begin{corollary}\label{cor: ap fourier}
Let $\Gamma$ be a discrete group satisfying the approximation property which acts freely on a compact Hausdorff space $X$.
If $M \subseteq C(X) \cpr_r \Gamma$ is a norm-closed $C(X)$-bimodule then
$$ M = \ol{\spn}\{ n \in N(M,C(X)) \} = \ol{\spn}\{ fu_g \colon f\in C(X),\ fu_g \in M\}. $$
\end{corollary}

\begin{proof}
Since $\Gamma$ acts freely on $X$, if $a \in M$ with $a\sim \sum_{g\in \Gamma} a_g u_g$, then $a_g u_g \in M$, by Proposition~\ref{prop: fourier coefs exist}.
The result follows now from Theorem~\ref{thm: ap fourier}.
\end{proof}

In \cite{CamSmi2019}, Cameron and Smith prove a spectral theorem for $A$-bimodules in $A \cpr_r \Gamma$, where $\Gamma$ is a discrete group satisfying the approximation property, and $A$ is a simple C$^*$-algebra.
There, the bimodules are determined by subsets of $\Gamma$.
When $\Gamma$ acts freely on a compact Hausdorff space $X$, we show that the $C(X)$-bimodules are determined by open subsets of the transformation groupoid $\Gamma \times X$.

For $g\in \Gamma$ we denote by $\pi_g$ the projection map on subsets $U \subseteq \Gamma \times X$ defined by
$$ \pi_g(U) = \{ x\in X \colon (g,x) \in U \}. $$

\begin{theorem}[Spectral Theorem for Bimodules]\label{thm: spectral}
Let $\Gamma$ be a discrete group satisfying the approximation property, which acts freely on a compact Hausdorff space $X$.
The map
$$ U \mapsto \ol{\spn}\{ f u_g \colon f\in C(X),\ \supp(f) \subseteq \pi_g(U) \}, $$
defines a one-to-one correspondence between open subsets $U \subseteq \Gamma \times X$ and closed $C(X)$-bimodules in $C(X) \cpr_r \Gamma$.
\end{theorem}

\begin{proof}
Let $U \subseteq \Gamma \times X$ be open.
Let
$$ N_U = \ol{\spn}\{f u_g \colon f \in C(X),\ \supp(f) \subseteq \pi_g(U) \} \subseteq C(X) \cpr_r \Gamma. $$
Then $N_U$ is a norm-closed $C(X)$-bimodule.  To show that the map
$\Phi \colon U \mapsto N_U$ is a surjective map from open subsets of
$\Gamma \times X$ and closed $C(X)$-bimodules in $C(X) \cpr_r \Gamma$
we construct the inverse map.

Let $N \subseteq C(X) \cpr_r \Gamma$ be a norm-closed $C(X)$-bimodule.
For each $g \in \Gamma$ define the set
$$ U_g = \bigcup\{ \supp(f) \colon f\in C(X),\ fu_g \in N\}. $$
Note that since $N$ is a $C(X)$-bimodule the set
$$ J_g = \{ f \in C(X) \colon fu_g \in N \} $$
is an ideal in $C(X)$.
Indeed $J_g = C_0(U_g)$.
Let
$$ U_N = \bigcup_{g\in \Gamma} \{g\} \times U_g. $$
Since each $U_g$ is open, $U_N$ is an open subset of $\Gamma \times X$.
Denote by $\Psi$ the map $\Psi \colon N \mapsto U_N$.

That $\Psi$ is the inverse of $\Phi$ follows from Corollary~\ref{cor: ap fourier}.
\end{proof}

Since, if $B$ is a C$^*$-algebra satisfying $C(X) \subseteq B \subseteq C(X) \cpr_r \Gamma$, then $B$ is a $C(X)$-bimodule, we get the following Corollary, which extends Theorem~\ref{thm: main}.

\begin{corollary}\label{cor: spectral}
Let $\Gamma$ be a discrete group satisfying the approximation property which acts freely on a compact Hausdorff space $X$.
Then the map
$$H \mapsto \ol{\spn}\{ f u_g \colon f\in C(X),\ \supp(f) \subseteq \pi_g(U) \}$$
defines a one-to-one correspondence between open subgroupoids $H \subseteq \Gamma \times X$ with $\{e\} \times X \subseteq H$ and intermediate C$^*$-algebras $B$ with $C(X) \subseteq B \subseteq C(X) \cpr_r \Gamma$.
\end{corollary}

\bibliography{intermediate}{}
\bibliographystyle{amsplain}

\end{document}